\documentclass[12pt]{article}
\usepackage{amsmath,amssymb,amsthm}
\newtheorem{theorem}{Theorem}
\newtheorem{corollary}[theorem]{Corollary}
\newtheorem{lemma}[theorem]{Lemma}

\numberwithin{equation}{section}
\DeclareMathOperator{\psin}{psin}

\def\P{\mathcal{P}}
\def\L{\mathcal{L}}

\def\A{\mathcal{A}}

\def\B{\mathcal{B}}

\newcommand{\eps}{\varepsilon}

\begin{document}

\title{Distinct spreads in vector spaces over finite fields}
\author{
Ben Lund\and
    Thang Pham
  \and
    Le Anh Vinh}
\date{}
\maketitle
\begin{abstract}
In this short note, we study the distribution of spreads in a point set $\P\subseteq \mathbb{F}_q^d$, which are analogous to angles in Euclidean space. 
More precisely, we prove that, for any $\eps > 0$, if $|\P| \geq (1+\eps) q^{\lceil d/2 \rceil}$, then $\P$ determines a positive proportion of all spreads.
We show that these results are tight, in the sense that there exist sets $\P \subset \mathbb{F}_q^d$ of size $|\P| = q^{\lceil d/2 \rceil}$ that determine at most one spread.
\end{abstract}
\section{Introduction}
Let $q=p^r$ be a large odd prime power, and $\mathbb{F}_q$ be a finite field of order $q$. 
For any two vectors $\mathbf{v}, \mathbf{u} \in \mathbb{F}_q^d$, we define the dot product
\[ \mathbf{v} \cdot \mathbf{u} = v_1 u_1 + v_2 u_2 + \ldots + v_d u_d.\]
We further define the distance between any two points $\mathbf{x}$ and $\mathbf{y}$ by  
\[\|\mathbf{x}-\mathbf{y}\|= (\mathbf{x} - \mathbf{y}) \cdot (\mathbf{x} - \mathbf{y}).\]
Although it is not a norm, the function $\|\mathbf{x}-\mathbf{y}\|$ has properties similar to the Euclidean norm (for example, it is invariant under orthogonal matrices and translations).

The Erd\H{o}s-Falconer distance problem asks for the minimum exponent $\alpha$ such that for any set $\P\subseteq \mathbb{F}_q^d$ with $|\P|\gg q^{\alpha}$, the number of distinct distances determined by $\P$ is at least $cq$ for some positive constant $c$.\footnote{Here and throughout,  $X \gg Y$ means that there exists $C>0$ such that $X\ge  CY$.} Bourgain, Katz, and Tao \cite{bourgain} considered a similar problem on the number of distinct distances determined by a set of points in $\mathbb{F}_q^2$.  Iosevich and Rudnev \cite{ir} proved that for any $\P\subseteq \mathbb{F}_q^d$, if $|\P|\ge 2q^{\frac{d+1}{2}}$, then all distances are determined by $\P$. The authors of \cite{hart} indicated that the exponent $(d+1)/2$ is the best possible in odd dimensions; the correct exponent is not known in even dimensions. 

Let $S_1$ be the unit sphere in $\mathbb{F}_q^d$, i.e. the set of points $\mathbf{x}\in \mathbb{F}_q^d$ with $||\mathbf{x}||=1$. If $\P$ is a subset in the unit sphere $S_1$, the authors of \cite{hart} proved that if $|\mathcal{P}|\ge Cq^{\frac{d}{2}}$ for some sufficiently large positive constant $C$, then there exists $c>0$ such that the number of distinct distances determined by points in $\P$ is at least $cq$. We note here that their result even can be stated in a stronger form which will be useful for our later applications. The interested reader can find more details in \cite[page 15]{hart}.
\begin{theorem}[\textbf{Hart et al}., \cite{hart}]\label{018}
For $\mathcal{P}\subseteq S_1$ in $\mathbb{F}_q^d$ with $d\ge 3$. Suppose that $|\mathcal{P}|\ge Cq^{\frac{d}{2}}$ for some  positive constant $C$, then the number of distinct distances determined by points in $\P$ is at least $\min\{q/2, Cq/4\}$.
\end{theorem}
In 2014, Bennett, Iosevich and Pakianathan \cite{iom} studied a generalization of the Erd\H{o}s distinct distance problem in vector spaces over finite fields. More precisely, they dealt with the distribution of classes of triangles in a point set in $\mathbb{F}_q^d$. They proved that for any $\P\subseteq\mathbb{F}_q^2$, if $|\P|\gg q^{7/4}$ then $\P$ determines at least a positive proportion of all congruence classes of triangles, where two triangles, denoted by $\Delta(\mathbf{a}_1, \mathbf{a}_2, \mathbf{a_3})$ and $\Delta(\mathbf{b}_1, \mathbf{b}_2, \mathbf{b}_3)$, are in the same congruence class if there exist an orthogonal matrix $M$ and a vector $\mathbf{z}\in \mathbb{F}_q^2$ such that $M\cdot \mathbf{a}_i+\mathbf{z}=\mathbf{b}_i$ for $1\le i \le 3$. The threshold $q^{7/4}$ was improved recently  to $q^{8/5}$ by Bennett et al. \cite{iosevich} by using a clever combination of Fourier analytic techniques and elementary results from group action theory. The authors of \cite{iosevich} also  gave a construction of a point set $\P=\A\times \B$ with $|\A|=q^{1/2-\eps'}$ and $|\B|=q$, and $\P$ determines at most  $cq^{3-\eps''}$ for $\eps''>0$. Iosevich \cite{iosevich2} conjectured that for any $\P\subset\mathbb{F}_q^2$, if $|\P|\ge  Cq^{3/2}$ for a sufficiently large constant $C$ then $\P$ determines a positive proportion of all congruence classes of triangles.  The interested reader can find more discussions and related problems in \cite{iom, iosevich, cover, hai}. There is also a series of papers dealing with similar problems, see for example \cite{ben, iosevich, iom,hanson, a10, igo, vinh, vinherdos}.

In this paper, we study a similar problem on the number of distinct \textit{spreads} determined by a point set in $\mathbb{F}_q^d$.  

For three points $\mathbf{a}, \mathbf{b}, \mathbf{c}\in \mathbb{F}_q^d$, the spread between two vectors $\overrightarrow{\mathbf{a}\mathbf{b}}$ and $\overrightarrow{\mathbf{a}\mathbf{c}}$ in $\mathbb{F}_q^d$, which is denoted by $S(\overrightarrow{\mathbf{a}\mathbf{b}}, \overrightarrow{\mathbf{a}\mathbf{c}})$ (or $S(\mathbf{b}, \mathbf{a}, \mathbf{c})$ for simplicity), is defined by 
\[S\left(\overrightarrow{\mathbf{a}\mathbf{b}}, \overrightarrow{\mathbf{a}\mathbf{c}}\right)=1-\frac{\left(\overrightarrow{\mathbf{a}\mathbf{b}}\cdot\overrightarrow{\mathbf{a}\mathbf{c}}\right)^2}{\|\overrightarrow{\mathbf{a}\mathbf{b}}\|\cdot\|\overrightarrow{\mathbf{a}\mathbf{c}}\|},\]
where $\|\overrightarrow{\mathbf{x}}\|=x_1^2+\cdots+x_d^2$.
If either term in the denominator is $0$, then $S(\overrightarrow{\mathbf{a}\mathbf{b}}, \overrightarrow{\mathbf{a}\mathbf{c}})$ is undefined.

It is clear that this definition is consistent with the square of the sine of the angle between two vectors $\overrightarrow{\mathbf{a}\mathbf{b}}$ and $\overrightarrow{\mathbf{a}\mathbf{c}}$ in Euclidean space
\[\sin(\theta)^2=1-\frac{\left(\overrightarrow{\mathbf{a}\mathbf{b}}\cdot\overrightarrow{\mathbf{a}\mathbf{c}}\right)^2}{\|\overrightarrow{\mathbf{a}\mathbf{b}}\|\cdot\|\overrightarrow{\mathbf{a}\mathbf{c}}\|}.\]

The following are some properties of the spread between two vectors $\overrightarrow{\mathbf{a}\mathbf{b}}$ and $\overrightarrow{\mathbf{a}\mathbf{c}}$:
\begin{enumerate}
\item[(i)] $S\left(\overrightarrow{\mathbf{a}\mathbf{b}},\overrightarrow{\mathbf{a}\mathbf{c}}\right)=S\left(r(\overrightarrow{\mathbf{a}\mathbf{b}}), s(\overrightarrow{\mathbf{a}\mathbf{b}})\right)$ for any $r, s\in \mathbb{F}_q^*$,
\item[(ii)] $S\left(\overrightarrow{\mathbf{a}\mathbf{b}},\overrightarrow{\mathbf{a}\mathbf{c}}\right)=S\left(\overrightarrow{\mathbf{a}\mathbf{c}},\overrightarrow{\mathbf{a}\mathbf{b}}\right)$,

\item[(iii)] $S\left(\overrightarrow{\mathbf{a}\mathbf{b}},\overrightarrow{\mathbf{a}\mathbf{c}}\right)=S\left(M\cdot \overrightarrow{\mathbf{a}\mathbf{b}},M\cdot \overrightarrow{\mathbf{a}\mathbf{c}}\right)$, where $M$ is an orthogonal matrix.
\end{enumerate}

In 2015, Bennett \cite{ben} made the first investigation on the number of distinct spreads determined by points in $\P\subseteq \mathbb{F}_q^d$. In particular, he obtained the following.
\begin{theorem}[Theorem 6.5, \cite{ben}]\label{bode}
Let $\mathcal{P}$ be a set of points in $\mathbb{F}_q^2$. If $|\mathcal{P}|\ge 2q-1$ then the number of distinct spreads determined by points in $\P$ is $q$. 
\end{theorem}

It is clear that Theorem \ref{bode} is sharp up to the coefficient of $q$, since the number of spreads determined by points in a line of $q$ points is at most one. For higher dimensional cases, Bennett \cite{ben} had an observation on a connection between spreads and distances:
\paragraph{A connection between spreads and distances on a sphere:} Suppose $\P_1$ is a subset in the unit sphere $S_1$, it is easily to check that $S\left(\overrightarrow{\mathbf{0}\mathbf{a}}, \overrightarrow{\mathbf{0}\mathbf{b}}\right)=S\left(\overrightarrow{\mathbf{0}\mathbf{c}}, \overrightarrow{\mathbf{0}\mathbf{d}}\right)$ with $\mathbf{a}, \mathbf{b}, \mathbf{c}, \mathbf{d}\in \P_1$ if and only if either $\|\mathbf{a}-\mathbf{b}\|=\|\mathbf{c}-\mathbf{d}\|$ or $\|\mathbf{a}-\mathbf{b}\|=\|\mathbf{c}+\mathbf{d}\|$. Thus if $\P_1$ determines a positive proportion of all distances then $\P_1$ determines a positive proportion of all spreads. Therefore if we have a set $\P\subset\mathbb{F}_q^d$ satisfying $|\P|\gg q^{\frac{d+2}{2}}$ then there exists a sphere of radius $t\ne 0$ such that $|S_t\cap \P|\gg q^{d/2}$. From the first property of spread, we may assume that $S_t$ is the unit sphere. It follows from Theorem \ref{018} that $S_1\cap \P$ determines a positive proportion of all distances, therefore $S_1\cap \P$ determines a positive proportion of all spreads. In other words, we have proved the following.
\begin{theorem}[Theorem 6.3, \cite{ben}]\label{dinhly1}
Let $\mathcal{P}$ be a set of points in $\mathbb{F}_q^d$, with $d\ge 3$. If $|\mathcal{P}|\gg q^{(d+2)/2}$ then $\mathcal{P}$ determines a positive proportion of all spreads.
\end{theorem}

We remark here that if $\P$ is a subset in the unit sphere $S_1$, the third listed author \cite{vinh} showed that for $\P\subseteq \mathbb{F}_q^3$ with $|\P|\gg q^{3/2}$, the number of occurrences of a fixed spread $\gamma$ among $\P$ is $\Theta\left(\frac{|\P|^2}{q}\right)$ if $1-\gamma$ is not a square in $\mathbb{F}_q$. 

The main purpose of this short note is to give sharp results on the number of distinct spreads determined by a large set in $\mathbb{F}_q^d$. Our main idea is to prove that for $\P \subset \mathbb{F}_q^d$ of size $|\P| = q^{\lceil d/2 \rceil}$,  there exists a point $\mathbf{p}$ in $\P$ such that it is incident to at least $(1-o(1))\frac{\alpha_{\eps}}{1+\eps}q^{d/2}$ lines spanned by $\P$. 

\paragraph{Statement of main results:}
Our first result gives us the number of distinct spreads determined by $\P\subseteq \mathbb{F}_q^d$ with $d$ even.
\begin{theorem}\label{spreads}
For any $\eps > 0$, there exists $c>0$ such that the following holds.
Let $\P$ be a set of points in $\mathbb{F}_q^d$ with $d\ge 2$ even. If $|\P| \geq  (1+\eps) q^{d/2}$, then the number of distinct spreads determined by $\P$ is at least $cq$.
\end{theorem}

If $\P$ be a subset in $\mathbb{F}_q^d$ with $d$ odd, then we can embed $\P$ in $\mathbb{F}_q^{d+1}$ with the last coordinate of $0$. Therefore, as a direct consequence of Theorem \ref{spreads}, we obtain the following result.
\begin{theorem}\label{odd}
For any $\eps > 0$, there exists $c>0$ such that the following holds.
Let $\P$ be a set of points in $\mathbb{F}_q^d$ with $d\ge 3$ odd. If $|\P| \geq (1+\eps) q^{(d+1)/2}$, then the number of distinct spreads determined by $\P$ is at least $cq$.
\end{theorem}
\paragraph{Sharpness of results:}
We show that the conditions on the size of $\P$ in Theorem \ref{spreads} and Theorem \ref{odd} are sharp.
\begin{theorem}\label{con1}
Suppose that either $q \equiv 1 \mod 4$ and $d$ even, or $q \equiv 3 \mod 4$ and $d \equiv 0 \mod 4$. 
Then there exists a subset $\P$ in $\mathbb{F}_q^d$ such that $|\P|=q^{d/2}$ and there is no spread determined by points in $\P$. 
\end{theorem}
\begin{theorem}\label{con2}
Suppose that either $q \equiv 1 \mod 4$ and $d$ odd, or $q \equiv 3 \mod 4$ and $d \equiv 1 \mod 4$.
Then there exists a subset $\P$ in $\mathbb{F}_q^d$ with such that $|\P|=q^{(d+1)/2}$ and the number of distinct spreads  determined by points in $\P$ is at most one.
\end{theorem}

The dependence of these constructions on the arithmetic properties of the underlying field is necessary, and the authors believe that it may be possible to improve Theorem \ref{spreads} in the case that $q \equiv 3 \mod 4$ and $d \equiv 2 \mod 4$.
See section $3$ for details.

The rest of this note is organized as follows: in Section $2$ we give a proof of Theorem \ref{spreads}, in Section $3$, we give proofs of Theorems \ref{con1} and \ref{con2}.
In Section $4$, we define spreads among more than three points, and leave the study of such spreads for future work.

\section{Proof of Theorem \ref{spreads}}

To prove Theorem \ref{spreads}, we make use of the following theorem due to the first listed author and Saraf in \cite{lund}.
\begin{theorem}[Corollary 5, \cite{lund}]\label{thm:becks}
For any $\eps >0$ and $\P \subseteq \mathbb{F}_q^d$ with $|\P| \geq (1+\eps)q^{d-1}$, the number of lines spanned by $\P$ is bounded below by $\alpha_{\eps}q^{2d-2}$, where $\alpha_\eps = \eps^2(1 + \eps + \eps^2)^{-1}$.
\end{theorem}

For the reader's convenience, we include a brief sketch of the proof of Theorem \ref{thm:becks}.
It is based on the expander mixing lemma for bipartite graphs ({\em e.g.} Lemma 8 in \cite{lund}).
The expander mixing lemma states that, if $G$ is a bipartite graph such that there is a large gap between the largest and second-largest eigenvalues of the adjacency matrix of $G$, and $V_L,V_R$ are sufficiently large subsets of the left and right vertices of $G$, then the number of edges between $V_L$ and $V_R$ is approximately the number that would be expected if $V_L,V_R$ were chosen uniformly at random.
To prove Theorem \ref{thm:becks}, we apply the expander mixing lemma to the incidence graph of points and lines in $\mathbb{F}_q^d$, with $V_L=\P$ and $V_R$ the set $L$ of lines that each contain at most one point of $\P$.
The fact that each line of $L$ contains at most one point of $\P$ gives an upper bound on the number of incidences between $\P$ and $L$, and a lower bound follows from the expander mixing lemma.
Combining these bounds gives Theorem \ref{thm:becks}.

By using Theorem \ref{thm:becks}, we are able to show in our following theorem that if the cardinality of $\P$ is much smaller than $q^{d-1}$, we still have many distinct lines spanned by $\P$. 
\begin{theorem}\label{thm5}
For any $0 < \eps < q-1$, let $\P \subseteq \mathbb{F}_q^d$ with $|\P| \geq (1+\eps)q^{k-1}$.
Then, the number of lines spanned by $\P$ is bounded below by $(1-o(1))\alpha_{\eps} q^{2k-2}$.
\end{theorem}

\begin{proof}
Let $\pi'$ be a uniformly random projection from $\mathbb{F}_q^d$ to $\mathbb{F}_q^k$.

Let $\mathbf{a},\mathbf{b}$ be two arbitrary distinct points in $\mathbb{F}_q^d$.
We claim that the probability that $\pi'(\mathbf{a}) = \pi'(\mathbf{b})$ is less than $q^{-k}$.
Note that, if $\pi'(\mathbf{a}) = \pi'(\mathbf{b})$, then $\pi'(\mathbf{a}-\mathbf{x}) = \pi'(\mathbf{b}-\mathbf{x})$ for an arbitrary translation vector $\mathbf{x}$.
Hence, we may without loss of generality assume that $\mathbf{a} = \mathbf{0}$.
Then, the quesiton of whether $\pi'(\mathbf{a}) = \pi'(\mathbf{b})$ reduces to the question of whether $\mathbf{b}$ lies in the kernel of $\pi'$, which is a uniformly random $(d-k)$-dimensional linear subspace.
This probability is $(q^{d-k}-1)/q^d < q^{-k}$.

Hence, by linearity of expectation, the expected number of pairs $\mathbf{a},\mathbf{b} \in P$ such that $\pi'(\mathbf{a}) = \pi'(\mathbf{b})$, denoted by $E_{coll}$, is $E_{coll} < \binom{|\P|}{2}q^{-k} = (1-o(1))(1+\eps)^2q^{k-2}/2$.
In particular, there exists a projection $\pi$ from $\mathbb{F}_q^d$ to $\mathbb{F}_q^k$ such that the number of collisions is at most $E_{coll}$.
By a Bonferroni inequality, the image $\pi(\P)$ of $\P$ has size at least $|\pi(\P)| \geq |\P| - E_{coll}$. Thus $|\pi(\P)|=(1-o(1))|\P|$. The conclusion of the theorem follows from Theorem \ref{thm:becks}, and the observation that $\pi(\P)$ does not span more lines than $\P$.
\end{proof}

\begin{corollary}\label{co1}
Let $\P$ be a set of points in $\mathbb{F}_q^d$ with $d$ even, and $\L$ be the set of spanned lines by $\P$. Suppose that $|\P|=(1+\eps) q^{d/2}$, $\eps>0$, then there exists a point $\mathbf{p}$ in $\P$ such that it is incident to at least $(1-o(1))\frac{\alpha_{\eps}}{1+\eps}q^{d/2}$ lines from $\L$.
\end{corollary}
\begin{proof}
It follows from Theorem \ref{thm5} that the number of lines spanned by $\P$ is bounded below by $(1-o(1))\alpha_{\eps}q^{d}$. By the pigeonhole-principle, there exists a point $\mathbf{p}$ in $\P$ such that it is incident to at least $(1-o(1))\frac{\alpha_{\eps}}{1+\eps}q^{d/2}$ lines, and the corollary follows.
\end{proof}
\paragraph{Proof of Theorem \ref{spreads}:}
By Corollary \ref{co1}, if $|\P| \geq (1+\eps)q^{d/2}$, then there exists a point $\mathbf{p}$ in $\P$ such that it is incident to at least $cq^{d/2}$ lines that are spanned by $\P$ for some positive constant $c$ depending on $\eps$. 

Suppose $d=2$.
Then, if $\sqrt{-1} \in \mathbb{F}_q$, then there are $q-1$ points of $\mathbb{F}_q^2$ at distance $0$ from $\mathbf{p}$, lying on a single isotropic line with slope $\sqrt{-1}$.
If $\sqrt{-1} \notin \mathbb{F}_q$, then there is no point distinct from $\mathbf{p}$ at zero distance from $\mathbf{p}$.
If $\mathbf{a}, \mathbf{b}, \mathbf{c} \in \P$ such are in three distinct, non-isotropic lines incident to $\mathbf{p}$, then an easy calculation shows that $S(\mathbf{a}, \mathbf{p}, \mathbf{b}) \neq S(\mathbf{a}, \mathbf{p}, \mathbf{c})$, which proves Theorem \ref{spreads} in the case $d=2$.

Suppose $d>2$.
We denote the set of lines incident to $\mathbf{p}$ by $\L'$. One can check that there exists a sphere $S_t$ of radius $t\ne 0$ such that $|S_t\cap \L'|\ge \frac{cq^{d/2}}{2}$. Without loss of generality, we assume that $\mathbf{p}=\mathbf{0}$ and $t=1$.  Theorem \ref{018} implies that $S_1\cap \L'$ determines a positive proportion of all distances. Thus Theorem \ref{spreads} follows from the connection between spreads and distances given in the introduction. $\square$

\section{Proofs of Theorems \ref{con1} and \ref{con2}}
In this section, we give constructions showing that our results are generally tight.

Vectors $\mathbf{v},\mathbf{u}$ are orthogonal if $\mathbf{v}\cdot \mathbf{u} = 0$.
If $\mathbf{v} \cdot \mathbf{v} = 0$, then $\mathbf{v}$ is isotropic.
Our construction relies on the existence of totally isotropic subspaces of dimension $d/2$.
Such subspaces have been considered before in a similar context, for example see \cite[Lemma 5.1]{hart}.

We consider two different cases, depending on whether $q \equiv 1 \mod 4$ or $q \equiv 3 \mod 4$.

\begin{lemma}\label{orthogonalIsotropic1}
If $q \equiv 1 \mod 4$ and $d$ is even, then there exist $d/2$ mutually orthogonal, linearly independent, isotropic vectors in $\mathbb{F}_q^d$.
\end{lemma}
\begin{proof}
Since $q \equiv 1 \mod 4$, we have that $i=\sqrt{-1}$ is an element of $\mathbb{F}_q$.
Let $\mathbf{v}_1 = (1,i,0,\ldots,0)$, $\mathbf{v}_2 = (0,0,1,i,0,\ldots,0)$, $\ldots$, $\mathbf{v}_{d/2} = (0,0,\ldots,1,i)$.
It is easy to verify that these vectors satisfy the conclusion of the lemma.
\end{proof}

\begin{lemma}\label{orthogonalIsotropic3}
If $q \equiv 3 \mod 4$ and $d \equiv 0 \mod 4$, then there exist $d/2$ mutually orthogonal, linearly independent, isotropic vectors in $\mathbb{F}_q^d$.
\end{lemma}
\begin{proof}
It is a classical fact that there exists an isotropic vector in $\mathbb{F}_q^d$ for $d\geq 3$; for example, it is a special case of the Chevalley-Warning theorem.
Let $a,b,c$ so that $(a,b,c)$ is isotropic.
Let $\mathbf{v}_1 = (a,b,c,0,\ldots,0)$, $\mathbf{v}_2=(0,-c,b,a,0,\ldots,0)$, $\mathbf{v}_3=(0,0,0,0,a,b,c,0,\ldots,0)$, $\ldots$, $\mathbf{v}_{d/2} = (0,\ldots,0,-c,b,a)$.
It is easy to verify that these vectors satisfy the conclusion of the lemma.
\end{proof}

Note that the assumption that $d \equiv 0 \mod 4$ is necessary in Lemma \ref{orthogonalIsotropic3}.
For example, it can be verified by a direct calculation that there is no set of three mutually orthogonal, linearly independent, isotropic vectors in $\mathbb{F}_3^6$.
It is an interesting open question whether Theorem \ref{spreads} is tight when $q \equiv 3 \mod 4$ and $d \equiv 2 \mod 4$.

\paragraph{Proof of Theorem \ref{con1}:}
Suppose $d=2m$ with $d \geq 1$.
If $q \equiv 3 \mod 4$, then further suppose that $m$ is even. 
Let $\P$ be the subspace spanned by the mutually orthogonal isotropic vectors $\mathbf{v}_1, \ldots, \mathbf{v}_m$ given by Lemma \ref{orthogonalIsotropic1} or \ref{orthogonalIsotropic3}.
 Since $\|\mathbf{v}_i\|=0$ for all $1\le i \le m$, it follows from the definition of spread that there is no spread determined by three vectors in $\P$. On the other hand, the size of $\P$ is $q^{d/2}$, which ends the proof of the theorem. $\square$

\paragraph{Proof of Theorem \ref{con2}:}
Suppose $d=2m+1$ with $m\ge 1$.
If $q \equiv 3 \mod 4$, then further suppose that $m$ is even. 
Let $\mathbf{v}_1, \ldots, \mathbf{v}_m$ be the mutually orthogonal, isotropic vectors given by Lemma \ref{orthogonalIsotropic1} or \ref{orthogonalIsotropic3} in the $(d-1)$-dimensional subspace of $\mathbb{F}_q^d$ consisting of points whose last coordinate is $0$, and let $\mathbf{v}_{m+1} = (0,0,\ldots,0,1)$.
Let $\P$ be the subspace spanned by $\mathbf{v}_1, \ldots, \mathbf{v}_{m+1}$.
The size of $\P$ is $q^{(d+1)/2}$. It is easy to check that the spread determined by any triple of points in $\P$ is either undefined or one. Thus the number of distinct spreads determined by $\P$ is at  most one. This concludes the proof of the theorem. $\square$

\section{Higher order spreads}
As mentioned in the introduction, spreads are analogous to the sine of angles in Euclidean space.
It may be interesting to consider a generalization of this notion to spreads between more than two difference vectors.
This is an analog of polar sine in Euclidean space.

In more detail, Let $\mathbf{v}_1, \ldots, \mathbf{v}_k$ be vectors in $\mathbb{F}_q^d$, with $d \geq k$.
Let $V = [\mathbf{v}_1, \ldots, \mathbf{v}_k]$ be the $d \times k$ matrix formed by concatenating the column vectors of $\mathbf{v}_1, \ldots, \mathbf{v}_k$.
Define
\[\psin(\mathbf{v}_1, \ldots{v}_k) = \frac{\det(V^TV)}{\prod_{i=1}^k \| \mathbf{v}_i \|}.\]
For points $\mathbf{x}_1, \ldots, \mathbf{x}_{k+1}$, we define the $k$-spread to be
\[S_k(\mathbf{x}_1, \ldots, \mathbf{x}_{k+1}) = \psin(\mathbf{x}_2 - \mathbf{x}_1, \ldots, \mathbf{x}_{k+1} - \mathbf{x}_1).\]
It is straightforward to verify that, for $k=2$, this matches the definition of spread given in the introduction.

A natural question is, for any fixed $k$ and $c>0$, how large a subset of $\mathbb{F}_q^d$ can determine fewer than $cq$ $k$-spreads?
We leave this for future work.

\section{Acknowledgments}
Research of the first listed author was supported by NSF grants CCF-1350572 and DMS-1344994.
The second listed author was partially supported by Swiss National Science Foundation grants 200020-162884 and 200020-144531. The research of the third listed author is funded by the National Foundation for Science and Technology Development Project. 101.99-2013.21.

\vspace{1cc}
\hfill\\
Department of Mathematics, \\
University of Georgia\\
USA\\
E-mail: lund.ben@gmail.com\\
\bigskip\\
Department of Mathematics,\\
EPF Lausanne\\
Switzerland\\
E-mail: thang.pham@epfl.ch\\
\bigskip\\
University of Education, \\
Vietnam National University\\
Viet Nam\\
E-mail: vinhla@vnu.edu.vn
\end{document}